\documentclass[reqno]{amsart}

\usepackage[utf8]{inputenc}
\usepackage[T1]{fontenc}
\usepackage{amsmath,amsthm, amssymb,amsfonts,nicefrac}
\usepackage{graphicx}
\usepackage{textcomp}
\usepackage[svgnames,hyperref]{xcolor}
\usepackage{graphpap}
\usepackage[backref=page, colorlinks, linktocpage, citecolor = blue, linkcolor = blue, urlcolor=blue]{hyperref}
\usepackage{enumitem}
\usepackage{caption}
\usepackage{subcaption}
\usepackage[all]{hypcap} 
\usepackage{xspace}

\DeclareFontFamily{U}{mathb}{\hyphenchar\font45}
\DeclareFontShape{U}{mathb}{m}{n}{
      <5> <6> <7> <8> <9> <10> gen * mathb
      <10.95> mathb10 <12> <14.4> <17.28> <20.74> <24.88> mathb12
      }{}
\DeclareSymbolFont{mathb}{U}{mathb}{m}{n}
\DeclareMathSymbol{\bigast}{1}{mathb}{"06}

\usepackage[cmtip]{xy}
\xyoption{pdf}
\xyoption{color}
\xyoption{all}

\usepackage{tikz}
\usetikzlibrary{cd,arrows}
\tikzset{>=stealth}
\tikzcdset{arrow style=tikz}
\tikzset{link/.style={column sep=1.8cm,row sep=0.16cm}}

\renewcommand{\to}{\longrightarrow}
\newcommand{\rat}{\dashrightarrow}

\newcommand{\Xl}{\mathcal X} 

\newcommand{\Dl}{\mathcal D} 
\newcommand{\Hl}{\mathcal H} 
\newcommand{\Ql}{\mathcal Q} 
\newcommand{\Cl}{\mathcal{CB}} 
\newcommand{\gammao}{\gamma_{\circ}}
\newcommand{\gammas}{\gamma_*}
\newcommand{\I}{\textup{1}\xspace}
\newcommand{\II}{\textup{2}\xspace}
\newcommand{\III}{\textup{3}\xspace}
\newcommand{\IV}{\textup{4}\xspace}
\newcommand{\V}{\textup{5}\xspace}
\newcommand{\VI}{\textup{6}\xspace}
\newcommand{\Ia}{\textup{I}\xspace}
\newcommand{\IIa}{\textup{II}\xspace}
\newcommand{\IIIa}{\textup{III}\xspace}
\newcommand{\IVa}{\textup{IV}\xspace}

\newcommand{\Glo}{\mathcal G_\circ} 
\newcommand{\Gls}{\mathcal G_*} 
\newcommand{\Jlo}{\mathcal J_\circ} 
\newcommand{\Jls}{\mathcal J_*} 
\newcommand{\Jl}{\mathcal J}

\newcommand{\id}{\text{\rm id}}

\newcommand{\Z}{\ensuremath{\mathbb{Z}}}

\newcommand{\F}{\ensuremath{\mathbb{F}}}

\newcommand{\R}{\ensuremath{\mathbb{R}}}
\newcommand{\C}{\ensuremath{\mathbb{C}}}
\newcommand{\p}{\ensuremath{\mathbb{P}}}

\newcommand{\FF}{\ensuremath{\mathbb{F}}}

\newcommand{\J}{\ensuremath{\mathcal{J}}}

\renewcommand\k{\mathrm{k}}
\renewcommand\phi{\varphi}
\DeclareMathOperator{\Aut}{Aut}
\DeclareMathOperator{\Autp}{Aut_\R(\p^2)}
\DeclareMathOperator{\Birp}{Bir_\R(\p^2)}
\DeclareMathOperator{\Bir}{Bir}

\DeclareMathOperator{\Id}{id}

\DeclareMathOperator{\Sar}{BirMori(\p^2)}

\def\dashmapsto{\mapstochar\dashrightarrow}

\newcommand{\pt}{\text{pt}}
\renewcommand{\phi}{\varphi}

\newtheorem{theorem}{Theorem}[section]
\newtheorem*{thm*}{Main Theorem}
\newtheorem{corollary}[theorem]{Corollary}
\newtheorem{lemma}[theorem]{Lemma}
\newtheorem{proposition}[theorem]{Proposition}
\theoremstyle{definition}
\newtheorem{definition}[theorem]{Definition}
\newtheorem{remark}[theorem]{Remark}
\newtheorem{example}[theorem]{Example}

\title{The real plane Cremona group is an amalgamated product}
\author{Susanna Zimmermann}
\thanks{During this work, the author was supported by the Swiss National Science Foundation, by Projet PEPS 2018 JC/JC and by ANR Project FIBALGA ANR-18-CE40-0003-01.}

\address{Susanna Zimmermann\\ Laboratoire angevin de recherche en mathématiques (LAREMA)\\ CNRS\\
Université d'Angers\\ 49045 Angers Cedex 1\\ France}
\email{susanna.zimmermann@univ-angers.fr}

\begin{document}
\maketitle
\thispagestyle{empty}

\begin{abstract}
We show that the real Cremona group of the plane is a non-trivial amalgam of two groups amalgamated along their intersection and give an alternative proof of its abelianisation.
\end{abstract}

\section{Introduction}

The plane Cremona group is the group $\Bir_\k(\p^2)$ of birational transformations of $\p^2$ defined over a field $\k$. 
For algebraically closed fields $\k$, the Noether-Castelnuovo theorem \cite{Cas} shows that $\Bir_\k(\p^2)$ is generated by $\Aut_\k(\p^2)$ and the quadratic map $[x:y:z]\dashmapsto[yz:xz:xy]$. It implies that the normal subgroup generated by $\Aut_\k(\p^2)$ is equal to $\Bir_\k(\p^2)$.
Furthermore, \cite[Appendix by {\sc Cornulier}]{C13} shows that $\Bir_\k(\p^2)$ is not isomorphic to a non-trivial amalgam of two groups. However, it is isomorphic to a non-trivial amalgam modulo one simple relation \cite{B12,L10,I84}, and it is isomorphic to a generalised amalgamated product of three groups, amalgamated along all pairwise intersections \cite{W92}. 
For $\k=\R$, the group $\Birp$ is generated by $\Autp$ and the two subgroups
\begin{align*}
&\Jls= \{f\in\Birp\mid f\ \text{preserves the pencil of lines through}\ [0:0:1]\ \}\\
&\Jlo=\{f\in\Birp\mid f\ \text{preserves the pencil of conics through}\ p_1,\bar{p}_1,p_2,\bar{p}_2\ \}
\end{align*}
where $p_1=[1:i:0],p_2=[0:1:i]$ \cite[Theorem 1.1]{BM14}.
Over $\C$, the analogon of $\Jlo$ is conjugate to $\Jls$ since a pencil of conics through four points in $\p^2$ in general position can be sent over $\C$ onto a pencil of lines through one point.

We define 
$\Glo\subset \Birp$ to be the subgroup generated by $\Autp$ and $\Jlo$, and by
$\Gls\subset\Birp$ the subgroup generated by $\Autp$ and $\Jls$.
Then $\Birp$ is generated by $\Gls\cup \Glo$, and the intersection $\Gls\cap \Glo$ contains the subgroup $\Hl$ generated by $\Autp$ and the involution $[x:y:z]\dashmapsto[xz:yz:x^2+y^2]$, which is contained $\Jlo\cap\Jls$.

\begin{theorem}\label{main thm}
We have $\Gls\cap \Glo=\Hl$, which is a proper subgroup of $\Glo$ and $\Gls$ and
\[\Birp\simeq\Glo\bigast_\Hl \Gls.\]
Moreover, both $\Gls$ and $\Glo$ have uncountable index in $\Birp$.
\end{theorem}

The action of $\Birp$ on the Bass-Serre tree associated to the amalgamated product yields the following:

\begin{corollary}\label{cor:tree}
Any of algebraic subgroup of $\Birp$ is conjugate to a subgroup of $\Gls$ or of $\Glo$.
\end{corollary}

For the finite subgroups of odd order Corollary~\ref{cor:tree} can also be verified by checking their classification in \cite{Y16}. 
An earlier version of this article used an explicit presentation of $\Birp$ and of $\Glo$ in terms of generators and generating relations, the first of which is proven in \cite{Z17} and the second was proven analogously in the earlier version of this article. 
The present version does not use either presentation. Instead, we look at the groupoid of birational maps between rational real Mori fibre spaces of dimension $2$. It contains $\Bir_\R(\p^2)$ as subgroupoid, is generated by Sarkisov links and isomorphisms and the elementary relations are a set of generating relations \cite{LZ17}, see also Theorem~\ref{thm:sarkisov}. This information is encoded in a square complex on which $\Birp$ acts \cite{LZ17}, see also \S~\ref{sec:squares}. 
This allows us to moreover provide a new proof of the abelianisation theorem of $\Birp$ given in \cite[Theorem 1.1(1)\&(3)]{Z17}: 

\begin{theorem}\label{thm:abel}
There is a surjective homomorphism of groups
$$\Phi\colon\Birp\to\bigoplus_{(0,1]}\Z/2\Z$$
such that its restriction to $\Jlo$ is surjective and  $\Gls\subset\ker(\Phi)=[\Birp,\Birp]$, where the right hand side is also equal to the normal subgroup generated by $\Autp$.
\end{theorem}

The proof of the main theorems rely on the description from \cite{LZ17} of elementary relations among so-called Sarkisov links, into which any real birational map of $\p^2$ decomposes \cite{Isk96}. In higher dimension over $\C$, the decomposition result is due to \cite{HMcK} (and \cite{Corti} for dimension $3$), and the description of elementary relations is generalised in \cite{BLZ}, where the authors moreover deduce that for $n\geq3$, the group $\Bir(\p^n_{\C})$ is a non-trivial amalgamated product of uncountably many factors along their common intersection. 
\bigskip

\noindent{\sc Acknowledgement:} I would like thank Anne Lonjou for asking me whether the real plane Cremona group is isomorphic to a generalised amalgamated product of several groups, and the interesting discussions that followed. 
I would also like to thank St\'ephane Lamy for discussions on the square complex, and
J\'er\'emy Blanc, Yves de Cornulier for helpful remarks, questions and discussions, and the referee for the very useful comments and suggestions.


\section{A square complex associated to the Cremona group} \label{sec:squares}

In this section we recall the square complex constructed in \cite{LZ17}. 

By a surface $S$, we mean a smooth projective surface defined over $\R$, and by $S_\C$ the same surface but defined over $\C$. We define the N\'eron-Severi space $N^1(S_\C)$ as the space of $\R$-divisors $N^1(S_\C):=\mathrm{Div}(S_\C)\otimes_{\Z}\R/\equiv$, where $\equiv$ is the numerical equivalence of divisors. The Galois group $\mathrm{Gal}(\C/\R)\simeq\Z/2\Z$ acts on $N^1(S_\C)$ and we denote by $N^1(S)$ the subspace of the invariant classes. Since we only consider surfaces with $S(\R)\neq\emptyset$ and $\C[S_\C]^*=\C^*$, $N^1(S)$ is also the space of classes of divisors defined over $\R$ (see for instance \cite[Lemma 6.3(iii)]{Sansuc}). The dimension of $N^1(S)$ is called {\em Picard rank} of $S$ and is denoted by $\rho(S)$. 
We can identify $N^1(S)$ with the space $N_1(S)$ of $1$-cycles defined over $\R$. 
For a surjective morphism $\pi\colon S\to B$ defined over $\R$, we denote by $N_1(S/B)\subset N_1(S)$ the subspace generated by curves contracted by $\pi$, and by $N^1(S/B)$ the quotient of $N^1(S)$ by the orthogonal of $N_1(S/B)$. We call $\rho(S/B)=\dim N^1(S/B)$ the {\em relative Picard} rank of $S$ over $B$.

If not stated otherwise, all morphisms are defined over $\R$, while curves and points contained in a real surface will be geometric curves and points, i.e. they are not necessarily defined over $\R$, but its $\mathrm{Gal}(\C/\R)$-orbit is.

\subsection{Rank $r$ fibrations and a square complex}\label{subsec:full complex}

\begin{definition}
Let $S$ be a smooth projective real surface, $B$ a point or a curve and $r \ge 1$ an integer. 
We say that a surjective morphism $\pi \colon S \rightarrow B$ with connected fibres is a \textit{rank $r$ fibration} if $\rho(S/B)=r$ and the anticanonical divisor $-K_S$ is $\pi$-ample. 
\end{definition}

The last condition means that for any curve $C$ contracted to a point by $\pi$, we have $K_S \cdot C < 0$.
The condition on the Picard number is that $\rho(S) = r$ if $B$ is a point, and $\rho(S) = r+1$ if $B=\p^1$. We may write $S/B$ instead of $\pi\colon S\to B$.

An isomorphism between two fibrations $S/B$ and $S'/B'$ is an isomorphism $S \stackrel{\simeq}{\rightarrow} S'$ such that there exists an isomorphism on the bases that makes the following diagram commute: 
\[
\begin{tikzpicture}[baseline= (a).base]
\node[scale=1](a) at (0,0){
\begin{tikzcd}
S\ar[r,"\simeq"] \ar[d,swap,"\pi"] & S'\ar[d,"\pi'"]\\
B\ar[r,"\simeq"] & B'
\end{tikzcd}
};
\end{tikzpicture}
\] 
In particular, $S/B$ and $S'/B'$ are fibrations of the same rank. 

The definition of a rank $r$ fibration puts together several
notions. 
If $B$ is a point, then $S$ is a real del Pezzo surface of Picard rank $r$.
If $B$ is a curve, then $S$ is a conic bundle of relative Picard rank $r$: a general fiber is isomorphic to a smooth plane real conic, and any singular fiber is the union of two $(-1)$-curves secant at one point. 
Remark also that being a rank 1 fibration is equivalent to being a (smooth) Mori fibre space of dimension $2$.

\begin{lemma}\label{lem:Mfs}
Let $S/B$ be a rank $r$ fibration and assume that $S$ is rational. 
\begin{itemize} 
\item If $r=1$, then $S/B$ is isomorphic to one of the following:
\begin{enumerate}
	\item $\p^2/\pt$,
	\item $\Ql=\{[w:x:y:z]\in\p^3\mid w^2=x^2+y^2+z^2\}/\pt$,
	\item $\F_0=\p^1\times\p^1/\p^1$ (the map is the second or first projection),
	\item $\F_n=\{([x:y:z],[u:v])\in\p^2\times\p^1\mid yv^n=zu^n\}/\p^1$, $n>0$,
	\item $\Cl_6=\{([w:x:y:z],[u:v])\in\Ql\times\p^1\mid wv=zu\}/\p^1$,
\end{enumerate}
\item If $r=2$, then $S/B$ is isomorphic to $\F_0/\pt$ or to the blow-up of a rank $1$ fibration in a real point of a pair of non-real conjugate points. 
\end{itemize}
\end{lemma}
\begin{proof}
The first statement is {\cite[Proposition 2.15]{BM14}}. Suppose that $S/B$ is a rank $2$ fibration. 
As $\rho(S/B)=2$, we can run the $\mathrm{Gal}(\C/\R)$-invariant two rays-game over $B$; there exist exactly two morphisms $\pi_i\colon S\to S_i$ with connected fibres and $\rho(S/S_i)=1$, $i=1,2$, such that $\pi$ factors through each $\pi_i$. 
\[
\begin{tikzpicture}[baseline= (a).base]
\node[scale=.75](a) at (0,0){
\begin{tikzcd}
&S\ar[dl,swap,"\pi_1"]\ar[dr,"\pi_2"]\ar[dd,"\pi"]&\\
S_1\ar[dr]&&S_2\ar[dl]\\
&B&
\end{tikzcd}
};
\end{tikzpicture}
\] 
If $S_1$ and $S_2$ are both curves, then $S_1\simeq S_2\simeq\p^1$ as $S$ is rational, and $S/S_1$ and $S/S_2$ are rank $1$ fibrations. From the classification of rank $1$ fibrations, it follows that $S/B\simeq\F_0/\pt$. 
Else, at least one of the $S_i$ is a surface, say $S_1$, and $\pi_1$ is a birational morphism.
If $C$ is a curve contracted by $S_1/B$, then both the exceptional divisor of $\pi_1$ and the strict transform $\tilde{C}$ of $C$ are contracted by $\pi$, hence $K_{S_1}\cdot C=K_S\cdot\pi^*(C)<0$. 
In particular, $S_1/B$ is a rank $1$ fibration. 
\end{proof}

For a rational surface $S$, we call \textit{marking} on a rank $r$ fibration $S/B$
a choice of a birational map $\phi\colon S \dashrightarrow \p^2$.
We say that two marked fibrations $\phi \colon S/B \rat \p^2$ and $\phi' \colon
S'/B' \rat \p^2$ are \textit{equivalent} if $\phi'^{-1} \circ \phi \colon S/B \rightarrow S'/B'$
is an isomorphism of fibrations.
We denote by $(S/B, \phi)$ an equivalence class under this relation.

If $S'/B'$ and $S/B$ are marked fibrations of respective rank $r'$ and $r$, we say that $S'/B'$ \textit{factorizes} through $S/B$ if the birational map $S' \rightarrow S$ induced by the markings is a morphism, and moreover there exists a (uniquely defined) morphism $B \rightarrow B'$ such that the following diagram commutes:
\[
\begin{tikzpicture}[baseline= (a).base]
\node[scale=1](a) at (0,0){
\begin{tikzcd}
S' \ar[rrr, "\pi'"] \ar[rd]&&&B'\\
&S\ar[r,"\pi"]& B\ar[ur]&
\end{tikzcd}
};
\end{tikzpicture}
\] 
In fact if $B' = \pt$ the last condition is empty, and if $B' \simeq \p^1$ it means that $S' \rightarrow S$ is a morphism of fibration over a common basis $\p^1$.
Note that $r'\geq r$. 


We define a 2-dimensional complex $\Xl$ as follows.
Vertices are equivalence classes of marked rank $r$ fibrations, with
$3 \ge r\ge 1$. 
We put an oriented edge from $(S'/B',\phi')$ to $(S/B,\phi)$ if $S'/B'$ factorizes through $S/B$.
If $r' > r$ are the respective ranks of $S'/B'$ and $S/B$, we say that the edge has type $r',r$.
For each triplets of pairwise linked vertices $(S''/B'',\phi'')$, $(S'/B', \phi')$, $(S/B, \phi)$ where $S''/B''$ (resp. $S'/B'$, $S/B$) is a rank $3$ (resp. $2$, $1$) fibration, we glue a triangle. 
In this way we obtain a 2-dimensional simplicial complex $\Xl$. 

\begin{lemma}{\cite[Lemma 2.3]{LZ17}} \label{lem:2 triangles}
For each edge in $\Xl$ of type $3,1$, there exist exactly two triangles that admit this edge as a side.  
\end{lemma}

By gluing all the pairs of triangles along edges of type 3,1, and keeping only edges of types 
3,2 and 2,1, we obtain a square complex that we still denote $\Xl$.
When drawing subcomplexes of $\Xl$ we will often drop part of the
information which is clear by context, about the markings, the equivalence
classes and/or the fibration.
For instance $S/B$ must be understood as $(S/B, \phi)$ for an implicit
marking $\phi$ and $\p^2$ as $\p^2/\pt$.

\subsection{Sarkisov links and elementary relations}

In this section we recall from \cite{LZ17} that the complex $\Xl$ encodes the notion of
Sarkisov links (or {\em links} for short), and of elementary relation between them.

A {\it Sarkisov link} $f\colon S\rat S'$ between two rank $1$ fibrations $S/B$ and $S'/B'$ is one of the following birational maps:
\begin{figure}[ht]
\[
\begin{tikzpicture}[baseline= (a).base]
\node[scale=.75](a) at (0,0){
\begin{tikzcd}
&S'\ar[dr]&\\
S\ar[dr]\ar[ur,dashed,"f"]&&B'\ar[dl]\\
&\pt&
\end{tikzcd}
};
\end{tikzpicture}
\begin{tikzpicture}[baseline= (a).base]
\node[scale=.75](a) at (0,0){
\begin{tikzcd}
&S''\ar[dl,swap,"\pi"]\ar[dr,"\pi' "]&\\
S\ar[dr]\ar[rr,dashed,"f"]&& S'\ar[dl]\\
&B&
\end{tikzcd}
};
\end{tikzpicture}
\begin{tikzpicture}[baseline= (a).base]
\node[scale=.75](a) at (0,0){
\begin{tikzcd}
&S'\ar[dr]\ar[dl,swap,"f"]&\\
S\ar[dr]&&B'\ar[dl]\\
&\pt&
\end{tikzcd}
};
\end{tikzpicture}
\begin{tikzpicture}[baseline= (a).base]
\node[scale=.75](a) at (0,0){
\begin{tikzcd}
S\ar[rr,"f",equal]\ar[d]&&S\ar[d]\\
B\ar[dr]&& B'\ar[dl]\\
&\pt&
\end{tikzcd}
};
\end{tikzpicture}
\]
\caption{From left to right: Sarkisov link $f$ of type I, II, III and IV.} 
\label{fig:links}
\end{figure}
\begin{itemize}[wide]
\item Link of type \Ia: $B=\pt$, $B'=\p^1$ and $f\colon S\rat S'$ is a blow-up of a real point or a pair of non-real conjugate points.
\item Link of type \IIa: $B=B'$ and there exist two blow-ups of a real point or non-real conjugate points $\pi\colon S''\to S$ and $\pi'\colon S''\to S'$ over $B$ such that $f=\pi'\circ\pi^{-1}$.
\item A link of type \IIIa is the inverse of a link of type I, i.e. $B=\p^1$, $B'=\pt$ and $f\colon S\to S'$ is the contraction of a real $(-1)$-curve or a pair of non-real conjugate $(-1)$-curves.
\item Link of type \IVa: $S= S'$ and $B,B'$ curves and $f$ is the identity on $S$. If $S$ is rational, then, by Lemma~\ref{lem:Mfs}, $S\simeq\F_0$, $B\simeq B'\simeq\p^1$ and $S/B,S/B'$ are the projections on the first and second factor.
\end{itemize}

Let $(S/B, \phi)$, $({S}'/B', \phi')$ be two marked rank 1 fibrations.
The induced birational map $S \rat S'$ is a \textit{Sarkisov link}
if and only if there exists a marked rank 2 fibration $S''/B''$ that factorizes through both $S/B$ and $S'/B'$.
\[
\begin{tikzpicture}[baseline= (a).base]
\node[scale=1](a) at (0,0){
\begin{tikzcd}
&S''/B''\ar[dl]\ar[dr]\\
S/B && S'/B'
\end{tikzcd}
};
\end{tikzpicture}
\]
Indeed, for links of type \Ia and \IIIa we take $S''/B''=S'/\pt$, for links of type \IIa we take $S''/B''=S''/B$, and for links of type \IVa, we take $S''/B''=S/\pt$. 
Equivalently, the vertices corresponding to $S/B$ and $S'/B'$ are at distance 2
in the complex $\Xl$, with middle vertex $S''/B''$. 

A \textit{path of links} from a rank $1$ fibration $(S/B, \phi)$ to another rank fibration $(S'/B',\varphi')$ is a path in $\Xl$ from  $(S/B, \phi)$ to $(S'/B', \phi')$ that passes only through edges of type $1,2$.

\begin{proposition}{\cite[Proposition 2.6]{LZ17}} \label{pro:from S3}
Let $(S'/B, \phi)$ be a marked rank 3 fibration.
Then there exist finitely many squares in $\Xl$ with $S'$ as a corner, and
the union of these squares is a subcomplex of $\Xl$ homeomorphic to a disk
with center corresponding to $S'$.
\end{proposition}

In the situation of Proposition \ref{pro:from S3}, by going around the boundary
of the disc we obtain a path of Sarkisov links whose composition is
an automorphism.
We say that this path is an \textit{elementary relation} between
links, coming from the rank $3$ fibration ${S'}/B$.
More generally, any composition of links that corresponds to a loop
in the complex $\Xl$ is called a \textit{relation} between Sarkisov
links.

\begin{theorem}{\cite[Proposition 3.14, Proposition 3.15]{LZ17}}\label{thm:sarkisov}
\begin{enumerate}
\item \label{sarkisov1} 
Any birational map between rank 1
fibrations is a composition of links and isomorphisms. In particular the complex
$\Xl$ is connected. 
\item \label{sarkisov2} 
Any relation between links is generated by elementary relations, and in particular $\Xl$ is simply connected.
\end{enumerate} 
\end{theorem}

The first part of Theorem~\ref{thm:sarkisov} can also be found in \cite[Theorem 2.5]{Isk96}. In fact, a relative version can be extracted from the classification of links in \cite[Theorem 2.6]{Isk96}.

\begin{proposition}\label{prop:relative_sarkisov}
Let $B$ be a curve and $S/B$ and $S'/B$ two rank $1$ fibrations. 
Any birational map $f\colon S\rat S'$ over $B$ is a composition of Sarkisov links of type \IIa over $B$. In particular:
\begin{enumerate}
\item\label{relative_sarkisov:1} Let $\eta\colon \p^2\rat\F_1$ be blow-up of $[0:0:1]$. Then any element of $\eta\Jls\eta^{-1}$ is a composition of isomorphisms and links of type~\IIa between Hirzebruch surfaces.

\item\label{relative_sarkisov:2} Let $\eta:=\eta_2\eta_1\colon\p^2\rat\Cl_6$ where $\eta_1\colon\p^2\rat\Ql$ is the link of type~\IIa blowing up $[1:i:0],[1:-i:0]$ and contracting the line passing through them, and $\eta_2\colon\Ql\rat\Cl_6$ is the blow-up of $\eta_1([0:1:i])$ and $\eta_1([0:1:-i])$. 
Then any element of $\eta\Jlo\eta^{-1}$ is a composition isomorphisms of $\Cl_6$ and links $\Cl_6/\p^1\rat\Cl_6/\p^1$ of type~\IIa.
\end{enumerate}
\end{proposition}

\subsection{Elementary discs}
We call the disc with center a rank $3$ fibration from Proposition~\ref{pro:from S3} an {\em elementary disc}. 
In this section, we classify them and therewith obtain an explicit list of elementary relations among rank $1$ fibrations.

\begin{lemma}\label{lem:no_interval}
Any edge of $\Xl$ is contained in a square. 
In particular, $\Xl$ is the union of elementary discs.
\end{lemma}
\begin{proof}
Any edge $e$ of $\Xl$ contains a vertex $S/B$ that is a rank $2$ fibration. 

Suppose that $e$ is of type $2,3$. From the two rays-game on $S$ we obtain an edge in $\Xl$ of type $1,2$ attached to $S/B$. This yields an edge of type $1,3$, which, by Lemma~\ref{lem:2 triangles}, is contained in a square. 

Suppose that $e$ is of type $1,2$. We now produce an edge of type $2,3$ attached to $S/B$, which will imply that $e$ is contained in a square as above by Lemma~\ref{lem:2 triangles}. 
If $S$ is a del Pezzo surface, then $2\leq\rho(S)\leq 3$ and so by Lemma~\ref{lem:Mfs}, $S$ is the blow-up of $\p^2$ or $\Ql$ in at most four points and hence $(K_{S})^2\geq4$. Therefore, the blow-up of $S$ in a real point or a pair of non-real conjugate points in general position yields a del Pezzo surface $S'$ and $S'/B$ is a rank $3$ fibration.
If $S$ is not a del Pezzo surface, it is the blow-up of a Hirzebruch surface $\F_n$, $n\neq0,1$. The blow-up of $S$ in a real point or a pair of non-real conjugate points not contained in the same fibre nor in any singular fibre of $S/\p^1$ yields a rank $3$ fibration $S'/\p^1$. In any case, we have obtained an edge of type $2,3$ attached to $S/B$.

By Proposition~\ref{pro:from S3}, any square is contained in an elementary disc, so $\Xl$ is the union of elementary discs. 
\end{proof}

Lemma~\ref{lem:no_interval} is not true in general for an arbitrary perfect field $\k$. Indeed, if $\k$ has an extension of degree $8$, then a Bertini involution whose set of base-points consists of one single Galois-orbit is a link of type II whose corresponding two edges in $\Xl$ are not contained in any square \cite[Lemma 4.3]{LZ17}. 

We now give some examples of elementary discs in $\Xl$ and will then prove that our list is exhaustive. 
In what follows, $X_d$ is a del Pezzo surface of degree $d=(K_{X_d})^2$. 

\begin{example}\label{disc:1}
We now describe a {\em disc of type \I}, pictured in Figure~\ref{fig:disc1}
Pick two general real points $p,q\in\Ql$. The surface $\Ql_{\C}$ is isomorphic to $\p^1_{\C}\times\p^1_{\C}$ and in $Q$, the two fibres through $p$ (resp. $q$) are a pair of non-real conjugate curves, whose union we denote by by $C_p\subset\Ql$ (resp. $C_q\subset\Ql$).
Let $X_7\to\Ql$ be the blow-up of $p$. 
Blowing up $q$ on $X_7$ yields morphisms $X_6\to X_7$. 
We can contract the strict transform $\tilde{C}_p$ of $C_p$ on $X_7$ onto a pair of non-real conjugate points in $\p^2$. By abuse of notation, we also denote by $\tilde{C}_p$ the strict transform on $X_6$. On $X_6$, $\tilde{C}_p$ and the exceptional divisor $E_q$ of $q$ are disjoint, and contracting first $\tilde{C}_p$ and then $E_q$ yields the square on the lower left. 
Analogously, we obtain the square on the right. Blowing up $p$ and $q$ in different order yields the middle square. 
The two fibrations $\F_1/\p^1$ lift to two fibrations $X_6/\p^1$. This yields the upper two squares. The curves $E_q$, $E_p$ and the geometric components of $\tilde{C}_p$ and $\tilde{C}_q$ are the only $(-1)$-curves on $X_6$. We have obtained a disc around the the rank $3$ fibration $X_6/\pt$.
\begin{figure}[ht]
\[
\begin{tikzpicture}[baseline= (a).base]
\node[scale=.75](a) at (0,0){
\begin{tikzcd}[cramped,sep=small]
\F_1/\p^1&&X_6/\p^1\ar[ll,swap,"\tilde{C}_p"]\ar[rr,"\tilde{C}_q"]&&\F_1/\p^1\\
\F_1/\pt\ar[u]\ar[dd,"q"]&&X_6/\pt\ar[ll,"\tilde{C}_p",swap]\ar[rr,"\tilde{C}_q"]\ar[u]\ar[dl,"q"]\ar[dr,"p",swap]&&\F_1/\pt\ar[u]\ar[dd,"p",swap]\\
&X_7/\pt \ar[dl,"\tilde{C}_p"]\ar[dr,"p"]&&X_7/\pt \ar[dr,"\tilde{C}_q",swap]\ar[dl,"q",swap]&\\
\p^2/\pt&& \Ql/\pt&&\p^2/\pt
\end{tikzcd}
};
\end{tikzpicture}
\]
\caption{A disc of type~\I.} 
\label{fig:disc1}
\end{figure}
\end{example}

\begin{remark}\label{rmk:Cl6}
The surface $\Cl_6$ is obtained from $\Ql$ by blowing up the pair of non-real conjugate points $[0:1:i:0],[0:1:-i:0]$. It is a del Pezzo surface of degree $6$ and hence has six $(-1)$-curves. 
The conic bundle $\Cl_6/\p^1$ has exactly two singular fibres, each of which has exactly two components, which are conjugate $(-1)$-curves. The remaining two $(-1)$-curves is a pair of non-real conjugate sections of $\Cl_6/\p^1$; they are the exceptional divisors of the morphism $\Cl_6\to\Ql$ blowing up $[0:1:i:0],[0:1:-i:0]$.

$\bullet$ The surface $\Cl_5$ obtained by blowing up one real point on $\Cl_6$ not contained in any $(-1)$-curve on $\Cl_6$ is a del Pezzo surface of degree $5$ and inherits the conic bundle structure from $\Cl_6$. It can also be obtained by blowing up two pairs of non-real conjugate points on $\p^2$ in general position: the surface $\Cl_5$ contains ten $(-1)$-curves, namely the exceptional divisors of the four blown-up points and the strict transform of the lines passing through any two of them. The latter form two pairs of non-real conjugate lines and two real lines. 
Contracting one of the real $(-1)$-curves yields a birational morphism $\Cl_5/\p^1\to\Cl_6/\p^1$.

$\bullet$ The surface $\Cl_4$ obtained by blowing up a pair of non-real conjugate points $r,\bar{r}$ in $\Cl_6$ not contained in any $(-1)$-curve, is a del Pezzo surface of degree $4$. It can be obtained by blowing up $\p^2$ in three pairs of non-real conjugate points $p,\bar{p}$, $q,\bar{q}$, $r,\bar{r}$, not all contained on one conic, composed by the contraction the strict transform of the line $L_{p\bar{p}}$ passing through $p,\bar{p}$. The $(-1)$-curves on $\Cl_4$ form eight pairs of non-real conjugate curves. Among them, the only curves which are disjoint from their conjugate are the images of the exceptional divisors of $q,\bar{q},r,\bar{r}$ and the strict transform of the conics passing through $p,\bar{p}$ and three of $q,\bar{q},r,\bar{r}$. 

$\bullet$ The surface obtained by blowing up two real points $r,s$ in $\Cl_6$ has a similar discription to $\Cl_4$, and can be obtained by blowing up points $p,\bar{p},q,\bar{q},r,s$ on $\p^2$, no three collinear, and contracting the line through $p,\bar{p}$. Its sixteen $(-1)$-curves form seven pairs of conjugate curves and two real curves, the latter two being the strict transforms of the conics through $p,\bar{p},q,\bar{q},r$ and $p,\bar{p},q,\bar{q},s$.  
\end{remark}

\begin{example}\label{disc:2}
We describe a {\em disc of type~\II}, pictured in Figure~\ref{fig:disc2}.
Pick two real points $p$ and $q$ on the conic bundle $\Cl_6/\p^1$ not contained in the same fibre. Let $S\to\Cl_6$ be the blow-up of $p$ and $S'\to\Cl_6$ the blow-up of $q$. Blowing up $q$ on $S$ and $p$ on $S'$ yields morphisms $X\to S$ and $X\to S'$ and the lower square. The rank $3$ fibration $X/\p^1$ has exactly four singular fibres: the pre-image of the fibre $f_p$ containing $p$, the pre-iamge of the fibre $f_q$ containing $q$, and the pre-image of the two singular fibres of $\Cl_6/\p^1$. We can contract the strict transforms $\tilde{f}_p$ of $f_p$ and $\tilde{f}_q$ of $f_q$ onto real points over $\p^1$. Contracting both yields a morphism $X/\p^1\to\Cl_6/\p^1$ and the remaining squares. By Remark~\ref{rmk:Cl6}, there are no other contractions from $X$, and so we have obtained a disc around the vertex $X/\p^1$.  
An analogous construction can be made for two pairs of non-real conjugate points $p,\bar{p}$ and $q,\bar{q}$, where no two of $p,\bar{p},q,\bar{q}$ are on the same fibre, and for a real point $p$ and a pair of non-real conjugate points $q,\bar{q}$, no two of which are on the same fibre. 
\begin{figure}[ht]
\[
\begin{tikzpicture}[baseline= (a).base]
\node[scale=.75](a) at (0,0){
\begin{tikzcd}[cramped,sep=small]
&&\Cl_6/\p^1&&\\
&S''/\p^1\ar[ur,"\tilde{f}_q"]\ar[dl,swap,"q"]&&S'''^/\p^1\ar[ul,swap,"\tilde{f}_p"]\ar[dr,"p"]&\\
\Cl_6/\p^1&&X/\p^1\ar[ur,"\tilde{f}_q"]\ar[ul,swap,"\tilde{f}_p"]\ar[dl,"q"]\ar[dr,"p"]&&\Cl_6/\p^1\\
&S/\p^1\ar[dr,"p"]\ar[ul,swap,"\tilde{f}_p"]&&S'/\p^1\ar[dl,swap,"q"]\ar[ur,"\tilde{f_q}"]&\\
&&\Cl_6/\p^1&&\\
\end{tikzcd}
};
\end{tikzpicture}
\]
\caption{A disc of type~\II, where $p$ (resp. $q$) denotes a real point or a pair of non-real conjugate points.} 
\label{fig:disc2}
\end{figure}
\end{example}

\begin{example}\label{disc:3}
We describe a {\em disc of type~\III}, pictured in Figure~\ref{fig:disc3}.
Pick two pairs of non-real conjugate points $p,\bar{p}$ and $q,\bar{q}$ in $\p^2$ that are not collinear. Let $X_7\to\p^2$ be the blow-up of $p,\bar{p}$. The blow-up of $q,\bar{q}$ on $X_7$ yields a morphism $\Cl_5\to X_7$ (see Remark~\ref{rmk:Cl6}). Blowing up $p,\bar{p}$ and $q,\bar{q}$ in different order yields the lower middle square. 
On $\Cl_5$, the strict transform $\tilde{L}_p$ of the line $L_p$ passing through $p,\bar{p}$ is a real $(-1)$-curve and disjoint from the exceptional divisor $E_q$ of $q$. Contracting $\tilde{L}_p$ and $E_q$ yields a birational morphism $\Cl_5\to\Ql$. The order of the contractions yields the lower left square. The contraction of $\tilde{L}_p$ preserves the conic bundle structure in $\Cl_5$, which yields the left upper square. We repeat the same construction with $q$ instead of $p$ and obtain the right side of the disc. The remaining complex $(-1)$-curves on $\Cl_5$ are the strict transforms of the lines passing through one of each pair $p,\bar{p}$ and $q,\bar{q}$. They form conjugate pairs of intersecting curves and hence cannot be contracted. We have thus obtained a disc around the the rank $3$ fibration $\Cl_5/\pt$. 
\begin{figure}[ht]
\[
\begin{tikzpicture}[baseline= (a).base]
\node[scale=.75](a) at (0,0){
\begin{tikzcd}[cramped,sep=small]
\Cl_6/\p^1&&&\Cl_5/\p^1\ar[lll,swap,"\tilde{L}_p"]\ar[rrr,"\tilde{L}_q"]&&&\Cl_6/\p^1\\
&\Cl_6/\pt\ar[ul]\ar[dl,swap,"q\bar{q}"]&&\Cl_5/\pt\ar[u]\ar[ll,swap,"\tilde{L}_p"]\ar[rr,"\tilde{L}_q"]\ar[dl,"q\bar{q}"]\ar[dr,swap,"p\bar{p}"]&&\Cl_6/\pt\ar[ur]\ar[dr,"p\bar{p}"]&\\
\Ql/\pt&& X_7/\pt\ar[ll,"\tilde{L}_p"]\ar[dr,"p\bar{p}"]&&X_7/\pt\ar[rr,swap,"\tilde{L}_q"]\ar[dl,swap,"q\bar{q}"]&&\Ql/\pt\\
&&&\p^2/\pt&&&
\end{tikzcd}
};
\end{tikzpicture}\]
\caption{A disc of type~\III.} 
\label{fig:disc3}
\end{figure}
\end{example}

\begin{example}\label{disc:4}
We describe a {\em disc of type~\IV}, pictured in Figure~\ref{fig:disc4}.
Pick two pairs of non-real conjugate points $p,\bar{p}$ and $q,\bar{q}$ in $\Ql$ in general position. Blowing up $p,\bar{p}$ yields a morphism $\Cl_6\to\Ql$. Blowing up the points $q,\bar{q}$ yields the birational morphism $\Cl_4/\p^1\to\Cl_6/\p^1$. Denote by $\tilde{C}_q\subset\Cl_4$ the strict transform of the pair of non-real conjugate fibres of $\Cl_6/\p^1$ containing $q$ and $\bar{q}$. Its contraction yields a morphism to $\Cl_6$ over $\p^1$. We have now obtained the two left squares in Figure~\ref{fig:disc4}. We can repeat the construction with $q$ instead of $p$ and obtain the right squares in Figure~\ref{fig:disc4}. The blow-ups of $p,\bar{p}$ and $q,\bar{q}$ commute, which yields the lower square. The upper square corresponds to the different orders of contraction of the disjoint pairs $\tilde{C}_p$ and $\tilde{C}_q$. 
By Remark~\ref{rmk:Cl6}, there are no other contractions possible from $\Cl_4$, hence we have obtained a disc around the rank $3$ fibration $\Cl_4/\pt$.
\begin{figure}[ht]
\[
\begin{tikzpicture}[baseline= (a).base]
\node[scale=.75](a) at (0,0){
\begin{tikzcd}[cramped,sep=small]
&&&\Ql/\pt&&&\\
\Cl_6/\p^1&&\Cl_6/\pt\ar[ur,"\tilde{C}_p"]\ar[ll]&&\Cl_6/\pt\ar[ul,"\tilde{C}_q",swap]\ar[rr]&&\Cl_6/\p^1\\
&\Cl_4/\p^1\ar[ul,"\tilde{C}_q"]\ar[dl,"q\bar{q}",swap]&&\Cl_4/\pt\ar[ll]\ar[rr]\ar[ul,"\tilde{C}_q",swap]\ar[ur,"\tilde{C}_p"]\ar[dr,"p\bar{p}",swap]\ar[dl,"q\bar{q}"]&&\Cl_4/\p^1\ar[ur,"\tilde{C}_q",swap]\ar[dr,"p\bar{p}"]\\
\Cl_6/\p^1&&\Cl_6/\pt\ar[dr,"p\bar{p}"]\ar[ll]&&\Cl_6/\pt\ar[dl,"q\bar{q}",swap]\ar[rr]&&\Cl_6/\p^1\\
&&&\Ql/\pt&&&\\
\end{tikzcd}
};
\end{tikzpicture}
\]\caption{A disc of type~\IV.} 
\label{fig:disc4}
\end{figure}
\end{example}

\begin{example}\label{disc:5}
We describe a {\em disc of type~\V}, pictured in Figure~\ref{fig:disc5}.
Pick two real points $p,q\in\p^2$. Let $\F_1\to\p^2$ be the blow-up of $p$. The pencil of lines in $\p^2$ through $p$ induces the fibration $\F_1/\p^1$, and the fibre containing $q$ is the strict transform of the line $L$ through $p$ and $q$. The blow-up $X_7\to\F_1$ of $q$ induces a fibration $X_7/\p^1$, and the contraction of the strict transform $\tilde{L}$ of $L$ yields a morphism $X_7\to\F_0$ preserving this fibration. This yields the left half of the disc. Exchanging the roles of $p$ and $q$ yields the right half, with the fibration $X_7/\p^1$ induced by the pencil of lines in $\p^2$ passing through $q$, and the lower middle square. The induced fibrations on $\F_0=\p^1\times\p^1$ are the two projections. On $X_7$ there are only three $(-1)$-curves, all of which are real curves, so our disc is complete.
\begin{figure}[ht]
\[
\begin{tikzpicture}[baseline= (a).base]
\node[scale=.75](a) at (0,0)
{\begin{tikzcd}[cramped,sep=small]
&\F_0/\p^1&&\F_0/\p^1&\\
&&\F_0/\pt\ar[ul]\ar[ur]&& \\
&X_7/\p^1\ar[uu,"\tilde{L}"]\ar[dl,"p",swap]&&X_7/\p^1 \ar[uu,"\tilde{L}",swap]\ar[dr,"q"]&\\
\F_1/\p^1&&X_7/\pt \ar[ur]\ar[ul]\ar[uu,"\tilde{L}"]\ar[dl,"q"]\ar[dr,"p",swap]&& \F_1/\p^1\\
&\F_1/\pt\ar[ul] \ar[dr,"p"]&&\F_1/\pt\ar[ur]\ar[dl,"q",swap]\\
&&\p^2/\pt&&
\end{tikzcd}
};
\end{tikzpicture}
\]
\caption{A disc of type~\V.} 
\label{fig:disc5}
\end{figure}
\end{example}

\begin{example}\label{disc:6}
A {\em disc of type~\VI} is constructed analogously to a disc of type~\II, and is pictured in Figure~\ref{fig:disc6}, but by using the conic bundle $\F_n/\p^1$, $n\geq0$, instead of $\Cl_6/\p^1$. As in Figure~\ref{fig:disc2}, the points $p$ and $q$ in Figure~\ref{fig:disc6} refer to real points or pairs of non-real conjugate points. Moreover, we have $m=n+1$ (resp. $m=n+2$) if $p$ is a real point (resp. a pair of non-real conjugate points $p$) contained in the exceptional section of $\F_n$, and $m=n-1$ (resp. $m=n-2$) otherwise. The same holds for $l$ and $q$ instead of $m$ and $p$, which yields the possible values of $k$. In particular, there are only Hirzebruch surfaces in such a disc. 
\begin{figure}[ht]
\[
\begin{tikzpicture}[baseline= (a).base]
\node[scale=.75](a) at (0,0){
\begin{tikzcd}[cramped,sep=small]
&&\F_{k}/\p^1&&\\
&S^2/\p^1\ar[ur,"\tilde{f}_q"]\ar[dl,swap,"q"]&&S^2/\p^1\ar[ul,swap,"\tilde{f}_p"]\ar[dr,"p"]&\\
\F_{m}/\p^1&&S^3/\p^1\ar[ur,"\tilde{f}_q"]\ar[ul,swap,"\tilde{f}_p"]\ar[dl,swap,"q"]\ar[dr,"p"]&&\F_{l}/\p^1\\
&S^2/\p^1\ar[dr,"p"]\ar[ul,swap,"\tilde{f}_p"]&&S^2/\p^1\ar[dl,swap,"q"]\ar[ur,"\tilde{f}_q"]&\\
&&\F_n/\p^1&&\\
\end{tikzcd}
};
\end{tikzpicture}
\]
\caption{A disc of type~\VI.} 
\label{fig:disc6}
\end{figure}
\end{example}

\begin{proposition}\label{lem:disc}\item
\begin{enumerate}
\item Any elementary disc in $\Xl$ is a disc of type \I, \dots, \VI. 
\item If two distinct elementary discs intersect, they do so either in exactly one vertex or in path of links. 
\item A disc of type $a\in\{\text{\II,\III,\IV}\}$ and a disc of type $b\in\{\text{\V,\VI}\}$ intersect in at most one vertex.
\end{enumerate}
\end{proposition}
\begin{proof}
The second and third claim follows from (1) and checking Exampes~\ref{disc:1}--\ref{disc:2}. 
Let $\Dl$ be an elementary disc and pick one of its squares $\mathcal{S}$. It contains a unique vertex that is a rank $1$ fibration $S/B$, which is one of the rank $1$ fibrations listed in Lemma~\ref{lem:Mfs}. The edges in $\mathcal{S}$ attached to $S/B$ correspond to blow-ups over $B$ of a real point or a pair of non-real conjugate points or to $\F_0/\pt\to\F_0/\p^1$ by Lemma~\ref{lem:Mfs}. So, the square $\mathcal{S}$ appears in a disc of type~\I, \dots, \VI. The disc $\Dl$ contains a unique rank $3$ fibration, so it is the rank $3$ fibration contained in $\mathcal{S}$. It follows that $\Dl$ is a disc of type~\I, \dots, \VI. 
\end{proof}

\section{The groups $\Gls$, $\Glo$ and $\Hl$, and a quotient of $\Birp$}

\subsection{The group $\Hl$.}

Recall that $\Gls\subset\Birp$ is the group generated by $\Autp$ and the group $\Jls$ of elements preserving the pencil of lines through $[1:0:0]$. 
The group $\Glo\subset\Birp$ is the group generated by $\Autp$ and the group $\Jlo$ of elements preserving the pencil of conics through the two pairs $[1:i:0],[1:-i:0]$ and $[0:1:i],[0:1:-i]$.

We denote by $\Hl\subset \Birp$ the subgroup generated by $\Autp$ and the quadratic involution $\sigma\colon [x:y:z]\dashmapsto[xz:yz:x^2+y^2]$. We have $\Hl\subseteq\Glo\cap\Gls$ since $\sigma\in\Jlo\cap\Jls$.

\begin{lemma}\label{lem:inters_path}
Let $g\in\Birp$. Then 
\begin{enumerate}
\item\label{inters_path:1} $g\in\Glo$ if and only if there exists a path of links from $(\p^2,\id)$ to $(\p^2,g)$ along discs of type~\I, \dots, \IV avoiding any vertices of the form $(\F_n,\varphi)$, $n\geq0$. 
\item\label{inters_path:2} $g\in\Gls$ if and only if there exists a path of links from $(\p^2,\id)$ to $(\p^2,g)$ along discs of type~\I, \V and \VI avoiding any vertices of the form $(\Ql,\varphi)$.
\item\label{inters_path:3}  $g\in\Hl$ if and only if there is a path of links from $(\p^2,\id)$ to $(\p^2,g)$ along discs of type~\I.
\end{enumerate}
\end{lemma}
\begin{proof}
(\ref{inters_path:1})
Let $g\in\Glo$ and write $g=g_{n+1}\alpha_ng_n\cdots\alpha_1g_1$ for some $g_i\in\Jlo$ and $\alpha_i\in\Autp$. Let $\eta\colon\p^2\rat\Cl_6$ the birational map from Proposition~\ref{prop:relative_sarkisov}(\ref{relative_sarkisov:2}). 
Then $\eta g_i\eta^{-1}$ is a birational map of the conic bundle $\Cl_6/\p^1$. 
The map $\eta$ corresponds to a path of links along discs of type~\I, \III and \IV. 
By Proposition~\ref{prop:relative_sarkisov}(\ref{relative_sarkisov:2}), $\eta g_i\eta^{-1}$ decomposes into links $\Cl_6\rat\Cl_6$ of type~\IIa over $\p^1$, corresponds to a path of links along discs of type~\II, \III, \IV. 
So, there exists a path of links from $(\p^2,\id)$ to $(\p^2,g)$ along discs of type~\I, \dots, \IV as claimed. 

Suppose there is a path of links from $(\p^2,\id)$ to $(\p^2,g)$ along discs of type~\I, \dots, \IV according to hypothesis. Then $g$ is the composition of links of type \IIa $\Cl_6/\p^1\rat\Cl_6/\p^1$ or $\p^2\rat\Ql$ blowing up a pair of non-real conjugate points and contracting the line passing through them, or of type~\Ia $\Ql\to\Cl_6$ or of type~\IIIa $\Cl_6\to\Ql$.
In particular, $g$ decomposes into automorphisms of $\p^2$ and elements of $\Jlo$, so is contained in $\Glo$. 

(\ref{inters_path:2}) is shown analogously to (\ref{inters_path:1}) but now $\Jls$ plays the role of $\Jlo$, the elementary discs \I,\V, \VI  play the role of the elementary discs of type ~\I, \II, \III, \IV, and the role of $\eta$ is played by the link $\p^2\rat\F_1$ of type~\I blowing up $[0:0:1]$, which corresponds to a path of links in an elementary disc of type \I. 

(\ref{inters_path:3}) 
The claim follows from the fact that $\sigma\colon [x:y:z]\dashmapsto[xz:yz:x^2+y^2]$ has a decomposition into links corresponding to the path of links $(\p^2/\pt,\id)\longleftarrow X_7/\pt\to\Ql/\pt\longleftarrow X_7/\pt\to(\p^2/\pt,\sigma)$ along a disc of type~\I. 
\end{proof}

\begin{lemma}\label{lem:intersection}
We have $\Hl=\Gls\cap\Glo$. 
\end{lemma}
\begin{proof}
Let $g\in\Gls\cap\Glo$. By Lemma~\ref{lem:inters_path}(\ref{inters_path:1}) there exists a path $\gamma_{\circ}$ of links from $(\p^2,\id)$ to $(\p^2,g)$ along discs of type~\I, \dots, \IV. 
By Lemma~\ref{lem:inters_path}(\ref{inters_path:2}) there exists a path $\gamma_*$ of links from $(\p^2,\id)$ to $(\p^2,g)$ along discs of type~\I, \V, \VI.  
Running from $(\p^2,\id)$ to $(\p^2,g)$ along $\gamma_{\circ}$ and then returning to $(\p^2,\id)$ via $\gamma_*$ yields a loop $\gamma$ in $\Xl$ at $(\p^2,\id)$. 
By Theorem~\ref{thm:sarkisov}(\ref{sarkisov2}), $\gamma$ is the boundary of a finite union $\Dl\subset\Xl$ of intervals and elementary discs. 
By Proposition~\ref{lem:disc}, the elementary discs in the $\Dl_i$ are discs of type ~\I, \dots, \VI, and we colour them as follows: the ones of type~\II, \III, \IV we colour blue, the ones of type~\V, \VI we colour red and the ones of type~\I we colour purple. 
Vertices or edges contained in the intersection of two discs of different colour are coloured purple, which is consistent with the intersection properties of elementary discs by Proposition~\ref{lem:disc}. 
By Lemma~\ref{lem:inters_path}(\ref{inters_path:3}) it suffices to construct a purple path of links from $(\p^2,\id)$ to $(\p^2,g)$ contained in $\Dl$. 
By Lemma~\ref{lem:inters_path}(\ref{inters_path:1})\&(\ref{inters_path:2}), the path $\gammas$ consists of red and purple intervals and $\gammao$ consists of blue and purple intervals, so that $\gammas\cap\gammao$ is a union of purple intervals. 
The closure of $\Dl\setminus (\gammas\cap\gammao)$ is a finite union of discs $\Dl_1,\dots,\Dl_n$ intersecting pairwise in at most one vertex. 
We can assume that $\gammas$ and $\gammao$ do not contain any loops, so that intersections of the $\Dl_i$ are vertices contained in $\gammas\cap\gammao$, which are in particular purple. 
For $i=1,\dots,n$, let $R_i\subset\Dl_i$ be the union of red elementary discs in $\Dl_i$ and $B_i\subset\Dl_i$ the union of blue elementary discs in $\Dl_i$. 
Then $R_i\cap B_i$ is a non-empty finite set of vertices. Since $\Dl_i$ is a disc, 
$P_i:=\Dl_i\setminus (R_i\cup B_i)$ is covered by purple discs and has a connected component containing $\Dl_i\cap\gammas\cap\gammao$.
\end{proof}

\subsection{The group $\Glo$}

We denote by $\sigma\in\Birp$ the quadratic map 
\[\sigma\colon [x:y:z]\dashmapsto[xz:yz:x^2+y^2].\] 
Recall that for any quadratic map $f\in\Birp$ with a pair of non-real conjugate base-points and one real base-point, there exist $\alpha,\beta\in\Autp$ such that $f=\alpha\sigma\beta$. 
Furthermore, if $\deg(\sigma\alpha\sigma)=2$, then $\sigma\alpha\sigma$ has one real and a pair of non-real conjugate base-points. 

\begin{remark}\label{rmk:deg}
Let $C\subset\p^2$ be a curve and let $f\in\Jls$ be of degree $d$. Let $p_0:=[0:0:1],p_1,\dots,p_{2d-2}$ be the base-points of $f$ and denote by $m_C(t)$ is the multiplicity of $C$ in a point $t$. If $f(C)$ is a curve, then 
\begin{align*}
\deg(f(C))
&=\deg(C)d-m_C(p)(d-1)-\sum_{i=1}^{2d-2}m_C(p_i)\\
&=\deg(C)+\sum_{i=1}^{d-1} \deg(C) - m_C(p)-m_C(p_{2i-1})-m_C(p_{2i})
\end{align*}
If $\deg(C)>\deg(f(C))$, then the sum in the last line is negative, 
which implies that there exists $j\in\{1,\dots,2d-2\}$ such that 
\[m_C(p)+m_C(p_{2j-1})+m_C(p_{2j})>\deg(C)\]
The same reasoning holds with $\geq$ instead of $>$.
\end{remark}

\begin{lemma}\label{lem:nontrivial1}
Any quadratic map in $\Hl$ has a real and a pair of non-real base-points. 
In particular, the map $\tau\colon [x:y:z]\dashmapsto[yz:xz:xy]$ is contained in $\J_*\setminus\Hl$, and so $\Hl\subsetneq\Gls$. 
\end{lemma}
\begin{proof}
Let $f\in\Hl$ be a quadratic map. 
We write $f=f_n\cdots f_1$, where $f_i=\alpha_ig_i\beta_i$ with $\alpha_i,\beta_i\in\Autp$ and $g_i\in\Jls$ of degree $\deg(g_i)>1$ with exactly one real base-point and all other base-points non-real points; we can do this because $\Hl$ is generated by $\Autp$ and $\sigma$, and in a first step, we can take $g_i=\sigma$ for all $i$.
For $i=1,\dots,n$, we denote by $\Lambda_i$ the linear system of the map $(f_i\cdots f_1)^{-1}$, and 
\[D:=\max\{\deg(\Lambda_i)\mid i=1,\dots,n\},\quad N:=\max\{i\mid \deg(\Lambda_i)=D\mid i=1,\dots,n\}.\] 
We now do induction on the lexicographically ordered pair $(D,N)$. 
Note that $D\geq2$, since $\deg(f)=2$, and that $f_{i+1}f_i$ has at most two real base-points for any $i=1,\dots,n-1$.

If $(D,N)=(2,1)$, then $f=f_1=\alpha_1\sigma\beta_1$ and we are done. Suppose that $(D,N)>(2,1)$. We will write $f_{N+1}f_N=\tau_m\cdots\tau_1$, where $\tau_i=\alpha_ig_i'\beta_i$ with $\alpha_i,\beta_i\in\Autp$ and $g_i'\in\Jls$ of degree $\deg(g_i')>1$ with exactly one real base-point and all other base-points  non-real points, and such that the pair $(D',N')$ associated to the sequence $f_n\cdots f_{N+2}\tau_m\cdots\tau_1f_{N-1}\cdots f_1$ is strictly smaller than $(D,N)$. 

If $D=2$, then $(D,N)=(2,n)$ and so $f_2f_1$ is of degree $\leq2$. If $f_2f_1$ is linear, we replace $f_3f_2f_1$ in the composition by $\tau:=f_3f_2f_1$. The sequence $f_n\cdots f_4\tau$ has pair $(D',N')=(2,n-2)$. If $\tau:=f_2f_1$ is of degree $2$, it has one real and two non-real conjugate base-points, and the sequence $f_n\cdots f_3\tau$ has associated pair $(D',N')=(2,n-1)$. 

Suppose that $D>2$. We denote by $m(t)$ the multiplicity of $\Lambda_N$ in a point $t$. Let $q_1$ (resp. $q_2$) be the real base-point of $f_N^{-1}$ (resp. $f_{N+1}$).

{\bf(a)} If $q_1=q_2$, then then $(\beta_{N+1}\alpha_N)\in\Jls$, and so $g_{N+1}\beta_{N+1}\alpha_Ng_N\in\Jls$ and the map $\tau:=f_{N+1}f_N=\alpha_{N+1}(g_{N+1}\beta_{N+1}\alpha_Ng_N)\beta_N$ has exactly real base-point, namely the real base-point one of $f_N$, and all its other base-points are non-real points. The sequence $f_n\cdots f_{N+2}\tau f_{N-1}\cdots f_1$ has associated pair $(D',N')<(D,N)$. 

{\bf(b)} Suppose that $q_1\neq q_2$. 
By Remark~\ref{rmk:deg} applied to a general member of the linear system $\Lambda_N$, there exist base-points $r_1,s_1$ (resp. $r_2,s_2$) of $f_N^{-1}$ (resp. $f_{N+1}$) such that
\begin{equation}\label{eq}
D\leq m(q_1)+m(r_1)+m(s_1),\quad D<m(q_2)+m(r_2)+m(s_2).
\end{equation}
For $i=1,2$, we can assume that $m(r_i)\geq m(s_i)$ and that $r_i$ is a point in $\p^2$ or is infinitely near $q_i$.

{\bf(b1)} Suppose that $m(q_1)\geq m(q_2)$. We first show that $r_2$ is a point in $\p^2$. If $r_2$ is infinitely near $q_2$, then $m(q_2)\geq m(r_1)+m(\bar{r}_1)=2m(r_1)\geq 2m(s_1)$. We obtain from inequalities (\ref{eq}) that $D<m(q_1)+2\cdot\frac{m(q_1)}{2}=2m(q_1)$, which is impossible. 
So, $r_2$ is a point in $\p^2$. 
From inequalities (\ref{eq}) we obtain that 
\[D<m(q_2)+2m(r_2)\leq m(q_1)+2m(r_2).\]
It follows that the triples $q_1,r_2,\bar{r}_2$ and $q_2,r_2,\bar{r}_2$ are not collinear. Thus, there exist quadratic maps $\rho_1,\rho_2\in\Birp$ with base-points $q_1,r_2,\bar{r}_2$ and $q_2,r_2,\bar{r}_2$, respectively. We have
\[\deg(\rho_i f_N\cdots f_1)=2D-m(q_i)-2m(r_2)<D,\quad i=1,2\]
and we can write $\tau_1:=\rho_1f_N=\gamma_1 g\delta_1$ and $\tau_3:=f_{N+1}\rho_2^{-1}=\gamma_2 g'\delta_2$  for some $\gamma_1,\gamma_2,\delta_1,\delta_2\in\Autp$ and $g,g'\in\J_*$ with only one real base-point and all other base-points real points. Furthermore, $\tau_2:=\rho_2\rho_1^{-1}$ is a quadratic map with a real and a pair of non-real conjugate base-points, so we can write $\tau_2=\gamma_3\sigma\delta_3$ for some $\gamma_3,\delta_3\in\Autp$. 
The situation is summarised in the following commutative diagram, where $\tilde{\Lambda}_i$ is the linear system of $(\rho_if_N\cdots f_1)^{-1}$, which is of degree  $\deg(\tilde{\Lambda}_i)<D$, $i=1,2$. 
\[
\begin{tikzpicture}[baseline= (a).base]
\node[scale=1](a) at (0,0){
\begin{tikzcd}[cramped,sep=small]
&\Lambda_N\ar[d,"\rho_1",dashed,swap]\ar[dr,"\rho_2",dashed,swap]\ar[drr,"f_{N+1}",dashed]&&\\
\Lambda_{N-1}\ar[ru,"f_N",dashed]\ar[r,dashed,swap,"\tau_1"]&\tilde{\Lambda}_1\ar[r,dashed,swap,"\tau_2"]&\tilde{\Lambda}_2\ar[r,dashed,swap,"\tau_3"]&\Lambda_{N+1}
\end{tikzcd}
};
\end{tikzpicture}
\]
The sequence $f_m\cdots f_{N+2}\tau_3\tau_2\tau_1f_{N-1}\cdots f_1$ has associated pair $(D',N')<(D,N)$. 

{\bf(b2)} If $m(q_2)>m(q_1)$ we proceed analogously to the case (b1) with $r_1$ instead of $r_2$. 
\end{proof}

\begin{lemma}\label{lem:Gs_uncountable_index}
The group $\Glo$ has uncountable index in $\Birp$.
\end{lemma}
\begin{proof}
Consider the map $\tau\colon[x:y:z]\dashmapsto[yz:xz:xy]$ and define the group
\[A:=\{[x:y:z]\mapsto[x+az:y+bz:z]\mid a,b\in\R\}\subset\Autp\]
Consider the map between sets $\psi\colon A\longrightarrow \Bir_\R(\p^2)/\Glo$, $\alpha\mapsto (\alpha\tau) \Glo$
We now prove that it is injective, which will yield the claim.
For all $\alpha\in A$ the map $\tau\alpha\tau$ is of degree $\leq2$, and $\tau\alpha\tau\in\Autp$ if and only if $\alpha=\Id$. If $\tau\alpha\tau$ is of degree $2$, it has three real base-points. 
Let $\beta,\gamma\in A$ such that $(\beta\tau)\Glo=(\gamma\tau)\Glo$. Then $\tau(\beta^{-1}\gamma)\tau\in\Glo$, and in particular $\tau(\beta^{-1}\gamma)\tau\in\Gls\cap\Glo=\Hl$, by Lemma~\ref{lem:intersection}.
Lemma~\ref{lem:nontrivial1} implies that $\tau\beta^{-1}\gamma\tau\in\Autp$ and hence $\beta^{-1}\gamma=\id$. It follows that $\varphi$ is injective. 
\end{proof}

\subsection{The group $\Gls$ and a quotient of $\Birp$}

\begin{remark}\label{rmk:std}
A link of type~\IIa of $\Cl_6$ blowing up a pair of non-real conjugate points is conjugate via the birational map $\eta\colon\p^2\rat\Cl_6$ from Proposition~\ref{prop:relative_sarkisov}(\ref{relative_sarkisov:2}) to an element $g\in\Jlo$ of degree $5$ with three pairs of non-real conjugate base-points, not all on one conic.  

Any two non-collinear pairs of non-real conjugate points in $\p^2$ can be sent by an automorphism of $\p^2$ onto $[1:i:0],[1:-i:0],[0:1:i],[0:1:-i]$. So, for any element of $f\in\Birp$ of degree $5$ with three pairs of non-real conjugate base-points not on one conic, there are $\alpha,\beta\in\Autp$ such that $\alpha f\beta\in\Jlo$.  We call $f$ a standard quintic transformation. See \cite[Example]{BM14} or \cite[\S1]{RV05} for equivalent definitions. 
\end{remark}

\begin{lemma}[{\cite[Lemma 3.19]{Z17}}]\label{deflem}
Let $f\in\Jl$, $\eta\colon\p^2\rat\Cl_6$ the birational map from Proposition~\ref{prop:relative_sarkisov}(\ref{relative_sarkisov:2}) and $\eta^{-1} f\eta=\varphi_n\cdots\varphi_1$ a decomposition into links of type~\IIa as in Proposition~\ref{prop:relative_sarkisov}(\ref{relative_sarkisov:2}).
For $j=1,\dots,s$, let $C_j$ be a (real or non-real) fibre of $\pi\colon\Cl_6\to\p^1$ contracted by $\phi_j$ and $\pi(C_j)=[a_j+ib_j:1]$ its image in $\p^1$. We define $v_j=1-\frac{|a_j|}{a_j^2+b_j^2}\in(0,1]$ if $b_j\neq0$, and $v_j=0$ otherwise. Then
\[
\psi\colon\Jlo\to\bigoplus_{(0,1]}\Z/2\Z,\quad f\mapsto\sum_{j=1}^s e_{v_j}
\]
is a surjective homomorphism of groups whose kernel contains all elements of $\Jlo$ of degree $\leq4$.
\end{lemma}

We now reprove \cite[Proposition 5.3]{Z17} by using the principle idea of \cite{LZ17} in the construction of a homomorphism $\Bir(\p^2_k)\to\bigast_I\Z/2\Z$ over a perfect field $k$. 

\begin{proposition}\label{prop:quotient}
The homomorphism $\psi\colon\Jlo\rightarrow\bigoplus_{(0,1]}\Z/2\Z$ lifts to a surjective homomorphism
\[\Psi\colon\Birp\to \bigoplus_{(0,1]}\Z/2\Z\]
whose kernel contains $\Gls$.
\end{proposition}
\begin{proof}
We denote by $\Sar$ the set of birational transformations between rank $1$ fibrations. It is a groupoid and contains $\Birp$ as subgroupoid, so it suffices to construct a homomorphism of groupoids 
\[\Psi\colon\Sar\to\bigoplus_{(0,1]}\Z/2\Z\]
whose restriction to its subgroup $\Jlo$ is $\psi$ and whose kernel contains $\Gls$. 

Let $\phi\colon\Cl_6\rat\Cl_6$ be a link of type~\IIa over $\p^1$ blowing up a pair of non-real conjugate points. Let $\eta_1\colon\p^2\rat\Ql$ and $\eta_2\colon\Ql\rat\Cl_6$ be the links from Proposition~\ref{prop:relative_sarkisov}(\ref{relative_sarkisov:2}) and $\eta=\eta_2\eta_1$ 
Then $\eta^{-1}\phi\eta\in\Jlo$ is a standard quintic transformation and we define $\Psi(\phi):=\psi(\eta^{-1}\phi\eta)$. 
For any other link $\phi\in\Sar$ and any isomorphism $\phi\in\Sar$ we define $\Psi(\phi):=0$. 
	To show that $\Psi$ is a homomorphism of groupoids, it remains to check that any relation between links and isomorphisms in $\Sar$ is sent onto zero. Since $\bigoplus\Z/2\Z$ is abelian, it suffices by Theorem~\ref{thm:sarkisov}(\ref{sarkisov2}) to check that elementary relations in $\Sar$ are sent onto zero by $\Psi$. 
Let $\phi_n\cdots\phi_1=\id$ be an elementary relation in $\Sar$. We can assume that one of the $\phi_i$ is a link of type~\IIa of $\Cl_6$ over $\p^1$ with a pair of non-real conjugate base-points.  
	The elementary relation $\phi_n\cdots\phi_1=\id$ corresponds to the boundary of an elementary disc in $\Xl$, and it is of type ~\II or type~\IV by Proposition~\ref{lem:disc} because one of the $\phi_i$ is a link of type~\IIa of $\Cl_6$ over $\p^1$ with a pair of non-real base-points. 
	
	If the disc is of type~\II, then $n=4$, $\eta^{-1}\varphi_i\eta\in\Jl$, $i=1,\dots,4$, and hence 
	\[\Psi(\varphi_4)\cdots\Psi(\varphi_1)=\psi(\eta^{-1}\varphi_4\eta)\cdots\psi(\eta^{-1}\varphi_1\eta)=\psi(\eta^{-1}\varphi_1\cdots\varphi_4\eta)=0.\] 
	
	If the disc is of type~\IV, we can assume up to conjugation that $\varphi_1,\varphi_3^{-1},\varphi_4,\varphi_6^{-1}\colon\Ql\rat\Cl_6$ are the links of type~\Ia in the relation. Up to automorphisms of $\Ql$ (which are sent onto zero by $\Psi$), we can furthermore assume that $\varphi_1=\varphi_3^{-1}=\eta_2$. 
	Then $\eta_1^{-1}\varphi_3\varphi_2\varphi_1\eta_1=\eta^{-1}\varphi_2\eta\in\Jlo$, and hence also $\eta_1^{-1}\varphi_6\varphi_5\varphi_4\eta_1\in\Jlo$. We obtain that 
	\[\Psi(\varphi_6)\cdots\Psi(\varphi_1)=\Psi(\varphi_5)\Psi(\varphi_2)=\psi(\eta_1^{-1}\varphi_6\varphi_5\varphi_4\eta_1)\psi(\eta^{-1}\varphi_2\eta)=0.\]
	This shows that $\Psi$ is a homomorphism of groupoids. By definition it coincides with $\psi$ on $\Jlo$, and its kernel contains $\Gls$ by Proposition~\ref{prop:relative_sarkisov}(\ref{relative_sarkisov:1}). 
\end{proof}

\begin{corollary}\label{cor:nontrivial2}
The group $\Gls$ does not contain any standard quintic transformations. In particular, $\Hl\subsetneq\Glo$
and the index of $\Gls$ in $\Birp$ is uncountable. 
\end{corollary}
\begin{proof}
Let $\Psi\colon\Birp\to\bigoplus_{(0,1]}\Z/2\Z$ be the homomorphism from Proposition~\ref{prop:quotient}. Its kernel contains $\Gls$ and hence also $\Hl$. For any standard quintic transformation $f\in\Glo$, we have $\Psi(f)\neq0$ by Remark~\ref{rmk:std}, Lemma~\ref{deflem} and Proposition~\ref{prop:quotient}. It follows that $f\notin\Gls$. 
Moreover, $\Psi$ induces a surjective map $\Birp/\Gls\to\bigoplus_{(0,1]}\Z/2\Z$ and hence the quotient $\Birp/\Gls$ is uncountable. 
\end{proof}

\section{Proofs of the main results}

\begin{proof}[Proof of Theorem~\ref{main thm}]
The group $\Birp$ is generated by the groups $\Gls$ and $\Glo$ by \cite[Theorem 1.1]{BM14}. To show that $\Birp$ is isomorphic to the amalgamated product $\Gls\bigast_{\Gls\cap\Glo}\Glo$, it suffices to show that any relation in $\Birp$ is the composition of conjugates of relations in $\Gls$ and relations in $\Glo$. 
By Theorem~\ref{thm:sarkisov}(\ref{sarkisov2}), any relation in $\Birp$ is generated by conjugates of elementary relations of links. 
An elementary relation is the boundary of an elementary disc, which are of \mbox{type~\I, \dots, \VI} by Proposition~\ref{lem:disc}. The boundary of a disc of type~\II, \III and \IV is conjugate a relation in $\Glo$,  the boundary of a disc of type~\V and \VI are conjugate to relations in $\Gls$, and the boundary of a discs of type~\I are conjugate to relations in $\Hl$ by Lemma~\ref{lem:inters_path}. 
We have $\Gls\cap\Glo=\Hl$ by Lemma~\ref{lem:intersection} and it is a proper subgroup of $\Glo$ and $\Gls$ by Lemma~\ref{lem:nontrivial1} and Corollary~\ref{cor:nontrivial2}. 
The groups $\Gls$ and $\Glo$ have uncountable index in $\Birp$ by Corollary~\ref{cor:nontrivial2}.
\end{proof}

\begin{proof}[Theorem~\ref{thm:abel}]
The homomorphism $\Psi\colon\Birp\rightarrow\bigoplus_{(0,1]}\Z/2\Z$ from Proposition~\ref{prop:quotient} coincides with the one given in \cite[Proposition 5.3]{Z17} since their restriction to the generating set $\Autp\cup\Jlo\cup\Jls$ of $\Birp$ coincide. The kernel of $\Psi$ is computed in \cite[\S6]{Z17} by using \cite[\S2--3]{Z17} and is equal to $[\Birp,\Birp]$ and to the normal subgroup generated by $\Autp$.
\end{proof}

\begin{proof}[Proof of Corollary~\ref{cor:tree}]
By Theorem~\ref{main thm}, the group $\Birp$ acts on the Bass-Serre tree $T$ of the amalgamated product $\Gls\ast_{\Gls\cap \Glo}\Glo$. Every element of $\Birp$ of finite order has a fixed point on $T$. It follows that every finite subgroup of $\Birp$ has a fixed point on $T$ \cite[\S I.6.5, Corollary 3]{S80}, and is in particular conjugate to a subgroup of $\Gls$ or of $\Glo$. 
For infinite algebraic subgroups of $\Birp$, it suffices to check the claim for the maximal algebraic subgroups of $\Birp$. 
By \cite[Theorem 1.1]{RZ16}, for any infinite maximal algebraic subgroup $G$ of $\Birp$ there is a birational map $\theta\colon\p^2\rat X$, where $X$ is one of the surfaces in the list below and $G=\theta^{-1}\Aut(X)\theta$:
\begin{enumerate}
\item $X=\p^2$, 
\item $X=\Ql$, 
\item $X=\FF_n$, $n\neq1$, 
\item $X$ is a del Pezzo surface of degree $6$ with a birational morphism $X\to \F_0$ blowing-up a pair of non-real conjugate points.
\item $X$ is a del Pezzo surface of degree $6$ with a birational morphism $X\to \F_0$ blowing-up two real points on $\FF_0$, 
\item There is a birational morphism $X\to \Cl_6$ over $\p^1$ of conic bundles blowing up $n\geq1$ pairs of non-real conjugate points on non-real fibres on the pair of disjoint non-real conjugate $(-1)$-curves of $\Cl_6$ (the exceptional divisors of the contraction $\Cl_6\to\Ql$),
\item There is a birational morphism $X\to \FF_n$ of conic bundles blowing up $2n\geq4$ points on the zero section of self-intersection $n$. 
\end{enumerate}
(1)\&(2)\&(3) We have $\Autp\subset \Hl=\Gls\cap\Glo$. The group $\Aut_\R(\Ql)$ is conjugate to a subgroup of $\Glo$, and for $n\geq0$, the group $\Aut_\R(\F_n)$ is conjugate to a subgroup of $\Gls$.

(4) The surface $X$ contains exactly three pairs of non-real conjugate disjoint $(-1)$-curves. The group $\Aut_\R(X)$ is generated by the lift of a subgroup of $\Aut_\R(\F_0)$ and two elements descending to birational maps of $\F_0$ preserving one of the two fibrations $\F_0/\p^1$ \cite[Proposition 3.5(2)\&(3)]{RZ16}. So, $\Aut_\R(X)$ is conjugate to a subgroup of $\Gls$. 

(5) The surface $X$ contains exactly six real $(-1)$-curves. Via the blow-down $\eta\colon X\rightarrow\p^2$ of three disjoint $(-1)$-curves, the group $\Aut_\R(X)$ is conjugate to a subgroup of $\Gls$.

(6) The group $\Aut_\R(X)$ is generated by the lift of a subgroup of $\Aut_\R(\Ql)$ and by elements descending to birational maps of $\Cl_6$ over $\p^1$ \cite[Propositio 4.5(1)\&(2)]{RZ16}. So, $\Aut_\R(X)$ is conjugate to a subgroup of $\Glo$.

(7) The group $\Aut_\R(X)$ is generated by the lift of a subgroups of $\Aut_\R(\F_n)$  and by elements descending to birational maps of $\F_n$ \cite[Proposition 4.8(1)\&(2)]{RZ16}. So, $\Aut_\R(X)$ is conjugate to a subgroup of $\Gls$.
\end{proof}

\bibliographystyle{abbrv}
\bibliography{biblio}

\begin{thebibliography}{10}

\bibitem{B12}
J.~Blanc.
\newblock Simple relations in the {C}remona group.
\newblock {\em Proc. Amer. Math. Soc.}, 140(5):1495--1500, 2012.

\bibitem{BLZ}
J.~Blanc, S.~Lamy, and S.~Zimmermann.
\newblock Quotients of higher dimensional cremona groups.
\newblock {\em \href{https://arxiv.org/abs/1901.04145}{Preprint
  arXiv:1901.04145}}, 01 2019.

\bibitem{BM14}
J.~Blanc and F.~Mangolte.
\newblock Cremona groups of real surfaces.
\newblock In {\em Automorphisms in birational and affine geometry}, volume~79
  of {\em Springer Proc. Math. Stat.}, pages 35--58. Springer, Cham, 2014.

\bibitem{C13}
S.~Cantat and S.~Lamy.
\newblock Normal subgroups in the {C}remona group.
\newblock {\em Acta Math.}, 210(1):31--94, 2013.
\newblock With an appendix by Yves de Cornulier.

\bibitem{Cas}
G.~Castelnuovo.
\newblock Le trasformazioni generatrici del gruppo cremoniano nel piano.
\newblock {\em Atti della R. Accad. delle Scienze di Torino}, (36):861--874,
  1901.

\bibitem{Corti}
A.~Corti.
\newblock Factoring birational maps of threefolds after {S}arkisov.
\newblock {\em J. Algebraic Geom.}, 4(2):223--254, 1995.

\bibitem{HMcK}
C.~D. Hacon and J.~McKernan.
\newblock The {S}arkisov program.
\newblock {\em J. Algebraic Geom.}, 22(2):389--405, 2013.

\bibitem{I84}
V.~A. Iskovskikh.
\newblock Proof of a theorem on relations in the two-dimensional {C}remona
  group.
\newblock {\em Uspekhi Mat. Nauk}, 40(5(245)):255--256, 1985.

\bibitem{Isk96}
V.~A. Iskovskikh.
\newblock Factorization of birational mappings of rational surfaces from the
  point of view of {M}ori theory.
\newblock {\em Uspekhi Mat. Nauk}, 51(4(310)):3--72, 1996.

\bibitem{L10}
S.~Lamy.
\newblock Groupes de transformations birationnelles de surfaces.
\newblock M{\'e}moire d'habilitation {\`a} diriger des recherches,
  Universit{\'e} Claude Bernarde Lyon 1, 2010.

\bibitem{LZ17}
S.~Lamy and S.~Zimmermann.
\newblock Signature morphisms from the cremona group over a non-closed field.
\newblock {\em accepted in J. of the European Math. Soc.}, 07 2017.

\bibitem{RZ16}
M.~Robayo and S.~Zimmermann.
\newblock Infinite algebraic subgroups of the real cremona group.
\newblock {\em Osaka J. of Math}, 55(4):681--712, 2018.

\bibitem{RV05}
F.~Ronga and T.~Vust.
\newblock Birational diffeomorphisms of the real projective plane.
\newblock {\em Comment. Math. Helv.}, 80(3):517--540, 2005.

\bibitem{Sansuc}
J.-J. Sansuc.
\newblock Groupe de {B}rauer et arithm\'etique des groupes alg\'ebriques
  lin\'eaires sur un corps de nombres.
\newblock {\em J. Reine Angew. Math.}, 327:12--80, 1981.

\bibitem{S80}
J.-P. Serre.
\newblock {\em Trees}.
\newblock Springer Monographs in Mathematics. Springer-Verlag, Berlin, 2003.
\newblock Translated from the French original by John Stillwell, Corrected 2nd
  printing of the 1980 English translation.

\bibitem{W92}
D.~Wright.
\newblock Two-dimensional {C}remona groups acting on simplicial complexes.
\newblock {\em Trans. Amer. Math. Soc.}, 331(1):281--300, 1992.

\bibitem{Y16}
E.~Yasinsky.
\newblock Subgroups of odd order in the real plane {C}remona group.
\newblock {\em J. Algebra}, 461:87--120, 2016.

\bibitem{Z17}
S.~Zimmermann.
\newblock The abelianisation of the real {C}remona group.
\newblock {\em Duke Mathematical Journal}, 167(2):211--267, 02 2018.

\end{thebibliography}

\end{document}